\documentclass[runningheads]{llncs}
\usepackage{bbding}
\usepackage{amsmath}
\usepackage{amsfonts}
\usepackage{enumerate}
\usepackage{epsfig}
\usepackage{hyperref}
\usepackage{color}
\usepackage{lipsum}

\newcommand{\tr}{\triangle}
\DeclareMathOperator{\aff}{aff}
\DeclareMathOperator{\cone}{cone}
\DeclareMathOperator{\port}{port}
\DeclareMathOperator{\rec}{rec}
\DeclareMathOperator{\cycle}{cycle}
\DeclareMathOperator{\ecycle}{ecycle}
\DeclareMathOperator{\row}{row}
\DeclareMathOperator{\spn}{span}
\DeclareMathOperator{\nul}{ker}

\newcommand{\1}{\mathbf{1}}
\newcommand{\0}{\mathbf{0}}
\newcommand{\R}{\mathbb R}
\newcommand{\Z}{\mathbb Z}
\newcommand{\D}{\mathbb D}

\newcommand{\zA}{\mathcal A}
\newcommand{\zB}{\mathcal B}
\newcommand{\zC}{\mathcal C}
\newcommand{\zD}{\mathcal D}
\newcommand{\zP}{\mathcal P}
\newcommand{\zO}{\mathcal O}

\newcommand{\zL}{\mathcal L}
\newcommand{\zF}{\mathcal F}

\newcommand{\zS}{\mathcal S}
\DeclareMathOperator{\col}{col}

\newcounter{claimcounter}

\newenvironment{myclaim}
  {\refstepcounter{claimcounter}%
   \medskip\par\noindent\textit{Claim~\theclaimcounter.}\ }
  {\par\medskip}

\newenvironment{myproof}
  {\setcounter{claimcounter}{0}%
   \begin{proof}}
  {\hfill$\qed$\end{proof}}

\newenvironment{proofofclaim}[1][]
  {\par\noindent\textit{Proof of Claim#1.}\ }
  {\hfill$\diamond$\par\medskip}

\begin{document}

\title{The Local Dyadic Conjecture\thanks{A preliminary version of this article appeared in: Proceedings of the 27th International Conference on Integer Programming and Combinatorial Optimization (IPCO 2026), Padova, Italy, June 17-19, 2026, pp. 144-157.}}
\author{
Mahtab Alghasi \and
Bertrand Guenin \and
Levent Tun\c{c}el \Envelope}
\institute{Department of Combinatorics and Optimization, University of Waterloo\\
\email{ \{malghasi, bguenin, levent.tuncel\} @uwaterloo.ca}}

\maketitle


\begin{abstract}
A rational with denominator a power of two is said to be dyadic. 
A long standing conjecture of Paul Seymour predicts that for an ideal clutter, the dual of the set covering problem always admits an optimal solution where all variables take dyadic values.
We present a local version of this conjecture, and motivate it by proving an important special case.
\end{abstract}

\keywords{Linear Programming, dyadic rationals, binary clutters, polyhedra}

\section{Introduction.}
\subsection{The dyadic conjecture.}
A family of sets $\zC$ over finite ground set $E(\zC)$ is a {\em clutter} if for every $S,S'\in\zC$ where $S\subseteq S'$ we have $S=S'$.
A clutter $\zC$ is {\em trivial} if $\zC=\emptyset$ or $\zC=\{\emptyset\}$.
Throughout this paper we shall assume that all clutters are non-trivial.
Given a clutter $\zC$, $M(\zC)$ denotes the $0,1$ matrix where the rows correspond to the characteristic vectors of members of $\zC$.
\footnote{$M(\zC)$ is only defined up to permutations.}
Consider a clutter $\zC$, and weights $w\in\Z_+^{E(\zC)}$ we consider the following primal-dual pair,
\begin{equation} \label{primal}
\min\{w^\top x:M(\zC)x\geq\1,x\geq\0\}, 
\end{equation}
\begin{equation} \label{dual}
\;\;\max\{\1^\top y: M(\zC)^\top y\leq w, y\geq\0\}. 
\end{equation}
Given a clutter $\zC$ we associate the polyhedron, 
\[
P(\zC):=\{x\geq\0:M(\zC)x\geq\1 \}.\]
The clutter $\zC$ is {\em ideal} if $P(\zC)$ is integral, i.e. every vertex (also known as extreme point) of $P(\zC)$ is in the integer lattice. 
In 1975, Paul Seymour (see \cite{schrijver03} 79.3e) proposed the following conjecture,
\begin{conjecture}\label{dyadic-cn}
Let $\zC$ be an ideal clutter. Then for every $w\in\Z_+^{E(\zC)}$ the linear program \eqref{dual} has an optimal solution that is dyadic.
\end{conjecture}
\noindent
Here, a rational number $\frac{r}{s}$ where $r,s\in\Z$ is {\em dyadic} if $s=2^q$ for some $q\geq 0$ where $q\in\Z$.
A vector is {\em dyadic} if every entry is dyadic.

Dyadic numbers are important for numerical computations because they have a finite binary representation, 
and therefore they can be represented exactly on a computer in floating-point arithmetic.
Checking whether a linear program admits a dyadic optimal solution can be done in polynomial time \cite{abdi24}. For an implementation see \cite{cornuejols25}.
For the set packing problems, integrality of the polytope, $\{x\geq\0:M(\zC)x\leq\1\}$ 
implies the existence of an integral solution to the dual problem
$\min\{\1^\top y:M(\zC)^\top y\geq w, y\geq\0\}$ for every $w\in\Z_+^{E(\zC)}$ \cite{chvatal75}. 
This affirms a one-to-one correspondence between perfect matrices and perfect graphs.
Conjecture~\ref{dyadic-cn} would establish a weak counterpart for set covering problems.

Conjecture \ref{dyadic-cn} is known to hold for a number of special cases, for instance,
(a) ideal clutters of odd cycles of graphs \cite{geelen02},
(b) clutters of $T$-cuts of grafts \cite{seymour78},
(c) clutters of $T$-joins of grafts \cite{abdi22},
(d) clutters of dicuts of a directed graph \cite{lucchesi78},
(e) clutters of dijoins of a directed graph \cite{guenin25} see also \cite{abdi25}.
For (a), (b) \eqref{dual} always have an $\frac12$-integral solution.
For (d) \eqref{dual} always has an integral solution. 
\subsection{A local strengthening of the Dyadic Conjecture.}
Consider a pointed polyhedron $P\subseteq\R^n$.
We say that a pair of vertices $v^1,v^2$ of $P$ are {\em adjacent} if both $v^1,v^2$ are contained in a single face of dimension $1$ of $P$.
We say that a vertex $v$ of a polyhedron $P$ is an {\em integral corner}, if $v$ and all vertices adjacent to $v$ are in $\Z^n$.
Consider now a clutter $\zC$ and weights $w\in\Z_+^{E(\zC)}$. 
Observe that $\min\{w^\top x:x\in P(\zC)\}$ always has an optimal solution and that the set of optimal solutions is a face $F$ of $P(\zC)$.
We say that $F$ is the {\em $w$-optimal face} of $P(\zC)$.
A set $B\subseteq E(\zC)$ is a {\em cover} of $\zC$ if $B\cap S\neq\emptyset$ for every $S\in\zC$.
A cover is $w$-minimum if it minimizes $w(B):=\sum_{j\in B}w_j$ among all covers.
We say that $w$ are {\em minimal} weights of $\zC$ if $w\in\Z_+^{E(\zC)}$ and every $j\in E(\zC)$ is contained in a $w$-minimum cover.

We predict the following.
\begin{conjecture}\label{corner-cn}
Let $\zC$ be a clutter and let $w$ be minimal weights of $\zC$.
If the $w$-optimal face of $P(\zC)$ contains an integral corner then \eqref{dual} has an optimal solution that is dyadic.
\end{conjecture}
\noindent
Observe that for an integral polyhedron, every vertex is an integral corner.
In particular, if $\zC$ is ideal then every vertex of $P(\zC)$ is an integral corner.
Furthermore, it is an easy exercise to show that it suffices to prove Conjecture \ref{dyadic-cn} for minimal weights only.
Hence, Conjecture \ref{corner-cn} implies Conjecture \ref{dyadic-cn}.
However, the former is strictly stronger, as there exists non-ideal clutters with integral corners.
Indeed, let $n\geq 3$ and denote by $\zC_n$ the clutter on the ground set  $\{e_1,\dots, e_n$, $\bar{e}_1,\dots,\bar{e}_n\}$ defined by 
\[
\zC_n
=
\bigl\{\{e_i,\bar{e}_i\} : i\in[n]\bigr\}
\cup
\bigl\{\{\bar{e}_i,\bar{e}_j\} : 1\leq i<j\leq n\bigr\}.
\]
Then, $x=(\frac{1}{2},\dots, \frac{1}{2})$ is a fractional vertex of the covering polyhedron $P(\zC_n)$, whereas  $v=\begin{bmatrix} \mathbf{0}_n |& \mathbf{1}_n \end{bmatrix}$ is an integral corner of $P(\zC_n)$. 

The condition that the optimal face contain an integral corner cannot be omitted.
Consider the case where $\zC$ are the lines of the Fano Matroid,~i.e.
\[
\zC=\{\{1,2,3\}, \{1,4,5\}, \{1,6,7\}, \{2,4,7\}, \{2,5,6\}, \{3,4,6\}, \{3,5,7\}\}.
\]
For minimal weights $w=\1$ the $w$-optimal face contains a unique point $\frac13\1$ and \eqref{dual} has a unique dual solution $\frac13\1$ as well.
It is also not possible to replace in Conjecture~\ref{corner-cn} the condition that ``{\em the $w$-optimal face contains an integral corner}" by ``{\em the $w$-optimal face contains an integral point}".
To see this consider the clutter $\zC$ with, 
\[
M(\zC)=
\begin{pmatrix}
0 &1 & 1 & 1 & 1 & 0 & 0 & 0\\
1 & 0 & 1 & 1 & 0 & 1 & 0 & 0\\
1 & 1 & 0 & 1 & 0 & 0 & 1 & 0\\
1 & 1 & 1 & 0 & 0 & 0 & 0 & 1\\
0 & 0 & 0 & 0 & 1 & 1 & 1 & 1
\end{pmatrix}.
\]
Let $w=\1$ and let $F$ denote the $w$-optimal face.
It can be readily checked that \eqref{dual} has a unique optimal solution $y=(\frac13,\frac13,\frac13,\frac13,\frac23)^\top$.
The vertices of $F$ are $v^1=(1,0,0,0,1,0,0,0)$, $v^2=(0,1,0,0,0,1,0,0)$, $v^3=(0,0,1,0,0,0,1,0)$, $v^4=(0,0,0,1,0,0,0,1)$.
We claim that none of $v^i$, $i\in[4]$ is an integral corner.
Because of symmetry it suffices to consider $v^1$, but $v^1$ is adjacent to $(\frac12,\frac12,\frac12,0,0,0,0,1)^\top$.
\subsection{A special case of Conjecture \ref{corner-cn}.}
The set of all inclusion-wise minimal covers of a clutter $\zC$ forms another clutter called the {\em blocker} of $\zC$.
We denote by $b(\zC)$ the blocker of $\zC$.
The blocking operation is idempotent, i.e., $b(b(\zC))=\zC$ \cite{Lehman64,Edmonds70}.
A clutter $\zC$ is said to be {\em binary} if for every $S\in\zC$ and every $B\in b(\zC)$ we have $|S\cap B|$ odd.
Cases (a)-(c) of Conjecture \ref{dyadic-cn} are all instances of binary clutters.
The restriction of Conjecture \ref{dyadic-cn} to binary clutters remains open and is therefore of particular interest.

A main corollary of this paper is the following special case of Conjecture~\ref{corner-cn}.
\begin{theorem}\label{main-th}
Let $\zC$ be a binary clutter and let $w$ be minimal weights of $\zC$.
If the $w$-optimal face of $P(\zC)$ contains a $(k,\ell)$-degenerate integral corner where $k(\ell+1)\leq 6$ then \eqref{dual} has an optimal solution that is dyadic.
\end{theorem}

A vertex $v$ of a full-dimensional polyhedron $P=\{x:Ax\leq b\}\subseteq \R^n$, where the system $Ax \leq b$ is irreducible,  is {\em $(k,\ell)$-degenerate} if
(i) $v$ is contained in $n+k$ faces of dimension $1$ of $P$, and
(ii) there are $n+\ell$ tight constraints of $Ax\leq b$ for $v$.

Observe that a $(0,0)$-degenerate vertex is non-degenerate.
For the case where the integral corner is non-degenerate, it can be readily shown that we in fact 
get an integral solution for \eqref{dual} as the left-hand-side of the tight constraints must be unimodular~\cite{personal}.

We observed that there exists non-ideal clutters with integral corners. We next show that this is also true for binary clutters.
Given a graph $G=(V,E)$ denote by $\zO(G)$ the clutter with groundset $E$ where the member of the $\zO(G)$ are the cycles with an odd number of edges.
It is well known that $\zO(G)$ is binary.
\begin{proposition}
For $n\geq 5$, $\zO(K_n)$ is a non-ideal clutter with integral corners.
\end{proposition}
\begin{proof}
For any graph $G$, $\zO(G)$ is ideal if and only if $K_5$ cannot be obtained from $G$ by first deleting edges and then contracting all edges in a cut \cite{Guenin01}.
It follows in particular, that $\zO(K_n)$ is non-ideal for every $n\geq 5$.
Denote the vertices of $K_n$ by $\{1,\ldots,n\}$ and let $B$ the set of edges in the clique induced by vertices $2,\ldots,n$.
Observe that $B\in b(\zO(K_n))$; hence, the characteristic vector $\bar{v}$ of $B$ is a vertex of $P(\zO(K_n))$ (\cite{cornuejols01} Remark 1.16).
We will show that $\bar{v}$ is an integral corner of $P:=P(\zO(K_n))$, thereby completing the proof.

Let $m=|E(\zO(K_n))|$ and observe that $m=\binom{n}{2}$.
The tight constraints at $\bar{v}$ are
(i) the non-negativity constraints $x_{1i}\geq 0$ for $i\in\{2,\ldots,n\}$; and
(ii) the odd-cycles constraints $x_{ij}+x_{i1}+x_{j1}\geq 1$ for distinct pairs $i,j\in\{2,\ldots,n\}$.
These $n-1+\binom{n-1}{2}=m$ constraints have linearly independent left-hand-sides, thus $\bar{v}$ is a non-degenerate vertex of~$P$.
Consequently, the $m$ faces of dimension $1$ incident with $\bar{v}$ in $P$ are obtained by relaxing one tight inequality (i) or (ii) at a time. 
Consider (ii) and let $i,j\in\{2,\ldots,n\}$. Relaxing $x_{ij}+x_{i1}+x_{j1}\geq 1$ yields a ray $\bar{v}+ \R_+ \mathbf{1}_{ij}$ of $P$.
Consider (i) and let $i\in\{2,\ldots,n\}$. 
Let $B_i$ be the set of edges in the clique formed by vertices $\{2,\ldots,n\}\setminus i$ together with the edge $1i$.
Observe that $B_i\in b(\zO(K_n))$, hence the characteristic vector $v^i$ of $B_i$ is a vertex of $P$.
Then $\bar{v}$ and $v^i$ are adjacent, since they are both contained in a face of dimension $1$ of $P$.
But then $v^2,\ldots,v^n$ are the only vertices adjacent to $\bar{v}$, in particular $\bar{v}$ is an integral corner.
\qed
\end{proof}

\subsection{Organization of the paper.}
Section \ref{sec-tech} presents a general result, Theorem \ref{main-tech}, and explains how Theorem \ref{main-th} follows as a corollary.
Section \ref{sec-binary} proves a result about binary clutters that might be of independent interest.
Section \ref{sec-small} proves that \eqref{dual} has a half-integer solution, when there are no more than $6$ active dual variables and the clutter is binary.
In Section \ref{sec-proof} we leverage these results, and a general tool on dyadic linear programs to prove Theorem \ref{main-tech}.
We end with closing remarks in Section \ref{sec-closing}.
\section{The general result.}\label{sec-tech}
Before we can proceed we require a number of definitions.
Consider a full-dimensional polyhedron $P=\{x:Ax\leq b\}$ with a vertex $v$ and a proper face $F$ that contains $v$.
Let $\bar{A}x\leq\bar{b}$ denote the implicit equalities for $F$.
A set $\zL$ of faces of dimension $1$ containing $v$ is {\em $(F,v)$-saturating} if no constraint of $\bar{A}x\leq\bar{b}$ is tight for every $L\in\zL$.
Observe that this is a property of the polyhedron, not of the system of linear constraints.

If $F,F'$ are faces of $P$ where $F'\supset F$ then $F'$ is a {\em superface} of $F$.
If in addition $\dim(F')=\dim(F)+1$ then $F'$ is an {\em immediate superface} of $F$.
Denote by $\zL'$ the set of faces of dimension $1$ containing $v$ and that 
contain a point distinct from $v$ in some immediate superface of $F$.
Then $F$ is {\em $v$-compatible} if $\zL'$ intersects every $(F,v)$-saturating set.
%
\begin{figure}[htb!]
\vspace{-0.1in}
\centering\includegraphics[scale=0.6]{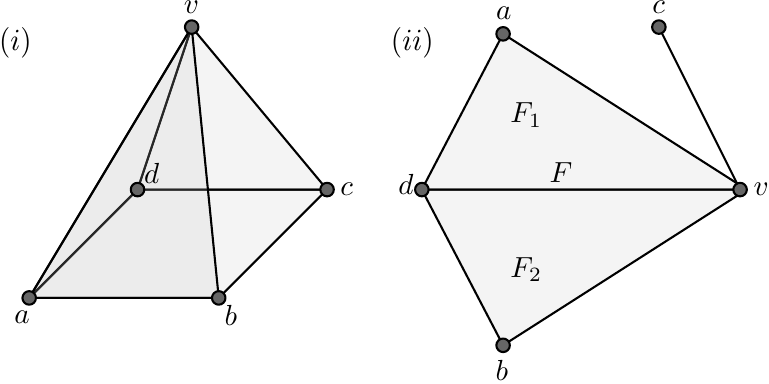}
\caption{Illustrating definitions.}\label{fig-def}
\end{figure}
We illustrate these definitions in Figure~\ref{fig-def}. 
For both (i) and (ii) denote by $L_a, L_b, L_c, L_d$ be the line segment between $v$ and $a,b,c,d$ respectively.
Consider (i).
The point $v$ is $(1,1)$-degenerate.
Pick $F=\{v\}$. 
Then the sets $\{L_a,L_c\}$ and $\{L_b,L_d\}$ are both $(F,v)$-saturating and every other $(F,v)$-saturating set is a superset of these sets.
The immediate superfaces of $F$ are $L_a,L_b,L_c$ and $L_d$.
The set $\zL'$ defined above is $\{L_a,L_b,L_c,L_d\}$.
In particular, $F$ is $v$-compatible. 
Indeed, whenever $F$ consists of a single point $v$, then it is always $v$-compatible.
Consider (ii). 
Pick $F=L_d$. 
Then the sets $\{L_a,L_b\}$ and $\{L_c\}$ are both $(F,v)$-saturating.
The immediate superfaces of $F$ are the triangles $F_1$ and $F_2$.
The set $\zL'$ defined above is $\{L_a,L_b\}$.
Since $\{L_c\}\cap\zL'=\emptyset$, $F$ is not $v$-compatible.

Suppose $P$ is full dimensional and determined by an irreducible system $Ax\leq b$.
Given a proper face $F$ of $P$, we denote by $\sigma(F,P)$ the number of tight constraints of $Ax\leq b$ for $F$.
Note, that $\sigma(F,P)$ is independent of the choice of $Ax\leq b$.
When clear from context, we write $\sigma(F)$ for $\sigma(F,P)$.

We can now state the main result of the paper.
\begin{theorem}\label{main-tech}
Let $\zC$ be a binary clutter and let $w$ be minimal weights of $\zC$.
Denote by $\bar{F}$  the $w$-optimal face of $P(\zC)$ and let $v$ be an integral corner of $P(\zC)$ contained in $\bar{F}$.
If every face $F\supseteq \bar{F}$ for which $\sigma(F,P(\zC))\geq 7$ is $v$-compatible then \eqref{dual} has an optimal solution that is dyadic.
\end{theorem}
We will show this implies Theorem~\ref{main-th}. We first require.
\begin{proposition}\label{dim-count}
Let $P$ be a pointed polyhedron of dimension $n$ and let $F$ be a face of $P$ of dimension $d$.
Then there exists at least $n-d$ immediate superfaces containing $F$.
\end{proposition}
\begin{proof}
We first consider the case when $P$ is a polytope.
Denote by $L(P)$ the face poset of $P$.
It is a graded lattice with rank function $r(G)=\dim(G)+1$ for faces $G$.
The interval $[F,P]$ of $L(P)$ is the face lattice of a polytope $P'$ of dimension $r(P)-r(F)-1=n-d-1$ (see \cite{Ziegler1995} Theorem 2.7).
$F=\emptyset$ in $P'$ and every immediate superface of $F$ is a vertex of $P'$.
Since $\dim(P')=n-d-1$, $P'$ must have at least $n-d-1+1$ vertices (equality holds for simplexes).
Hence, $F$ is contained in at least $n-d$ immediate superfaces, as required.

Consider now the general case where $P$ need not be a polytope.
Pick $M\in\Z_+$ large enough such that every vertex of $P$ is contained in the interior of the box $B:=\{x:-M\1 \leq x \leq M\1\}$.
Describe $P$ by an irreducible set of inequalities $Ax\leq b$ where $A$ has index rows $[m]$.
If $G$ is a face of $P$, then there exists $I_G\subseteq[m]$ such that $G=P\cap\{x:\row_i(A)x=b_i, i\in I_G\}$. 
Moreover, if $G$ is a face of $P\cap B$ where $G\cap F \neq \emptyset$, then none of the constraints of $B$ are tight for $F$ as $F$ must contain a vertex of $P$.
Hence, there exists a bijection between faces $G\supseteq F$ of $P$ and faces $G'\supseteq F$ of $P\cap B$.
Therefore, we can count the number of superfaces of $F$ in $P\cap B$ rather than $P$ and use the aforementioned result for polytopes.
\qed
\end{proof}
\noindent
Next we tie degeneracy to the number of active dual variables.
\begin{proposition}\label{reduction}
Let $P\subset\R^n$ be a full dimensional polyhedron.
Let $v$ be a $(k,\ell)$-degenerate vertex of $P$ where $k(\ell+1)\leq\alpha$ for some $\alpha\in\Z_+$.
Let $F\ni v$ be a face of $P$ with $\sigma(F,P)>\alpha$.
Then $F$ is $v$-compatible.
\end{proposition}
\begin{proof}
Let $Ax\leq b$ be an irreducible representation of $P$ and let $\bar{A}x\leq\bar{b}$ denote the tight constraints for $v$.
Let $\zL$ be the set of all faces of dimension $1$ of $P$ containing $v$.
We say that $\zL'\subseteq\zL$ is {\em non-tight} for a system of linear inequalities, 
if for each constraint there exists a point in some $L\in\zL'$ for which the constraint is not tight.
\begin{myclaim}\label{r-cl1}
Let $A'x\leq b'$ be a subsystem of $\bar{A}x\leq\bar{b}$ and let $L\in\zL$.
Then at most $\ell+1$ constraints of $A'x\leq b'$ are non-tight for $\{L\}$.
\end{myclaim}
\begin{proofofclaim}
Since $L$ is a face of dimension $1$ and $P\subseteq\R^n$, there are $\geq n-1$ constraints of $\bar{A}x\leq\bar{b}$ that are tight for $L$.
As $v$ is $(k,\ell)$-degenerate, $\bar{A}x\leq\bar{b}$ has at most $n+\ell$ constraints
Hence, at most $n+\ell-(n-1)=\ell+1$ of the constraints of $\bar{A}x\leq\bar{b}$ are non-tight for $\{L\}$.
As $A'x\leq b'$ is a subsystem of $\bar{A}x\leq\bar{b}$ the result follows.
\end{proofofclaim}
\noindent
Let $F_1,\ldots,F_p$ denote the immediate superfaces of $F$ and let $\zF=F\bigcup_{i\in[p]}F_i$.
\begin{myclaim}\label{r-cl2}
At least $n$ of the faces in $\zL$ intersect $\zF\setminus\{v\}$.
\end{myclaim}
\begin{proofofclaim}
Let $d=\dim(F)$. Then $p\geq n-d$, by Proposition \ref{dim-count}.
Since $F$ has dimension $d$, at least $d$ of the faces in $\zL$ intersect $F\setminus\{v\}$.
For each $i\in[p]$, at least one face $L_i\in\zL$ intersect $F_i\setminus F$.
Pick distinct $i,j\in[p]$. As $F_i,F_j$ are immediate superfaces and $F_i\neq F_j$, we have $L_i\neq L_j$.
Hence, we have $d+p\geq d+(n-d)\geq n$ distinct faces of $\zL$ that intersect $\zF\setminus\{v\}$.
\end{proofofclaim}
\noindent
Let $A'x\leq b'$ be the set of tight constraints of face $F$.
As $v\in F$, $A'x\leq b'$ is a subsystem of $\bar{A}x\leq\bar{b}$.
We may assume that there exists a $(F,v)$-saturating set $\zL'\subseteq\zL$ where none of the faces in $\zL'$ intersect $\zF\setminus\{v\}$, for otherwise $F$ is $v$-compatible.
Since $v$ is $(k,\ell)$-degenerate, $|\zL|\leq n+k$.
By Claim~\ref{r-cl2}, $n$ of the faces in $\zL$ intersect $\zF\setminus\{v\}$, hence $|\zL'|\leq n+k-n\leq k$.
For each $L\in\zL'$, by Claim~\ref{r-cl1}, at most $\ell+1$ of constraints $A'x\leq b'$ are non-tight for $\{L\}$.
Hence, at most $k(\ell+1)$ of constraints of $A'x\leq b'$ are non-tight for $\zL'$.
By hypothesis, $\sigma(F,P)>\alpha$, i.e. $A'x\leq b'$ has more than $\alpha$ constraints.
Since $\zL'$ is $(F,v)$-saturating $A'x\leq b'$ is non-tight for $\zL'$.
It follows that $k(\ell+1)>\alpha$, contradicting $k(\ell+1)\leq\alpha$.
\qed
\end{proof}
\begin{proposition}
Theorem~\ref{main-tech} implies Theorem~\ref{main-th}.
\end{proposition}
\begin{proof}
Let $\alpha=6$. 
By hypothesis of Theorem~\ref{main-th}, $v$ is $(k,\ell)$-degenerate for $k(\ell+1)\leq \alpha$.
It follows from Proposition~\ref{reduction} that every face $F\supseteq\bar{F}$ for which $\sigma(F)\geq 7$ is $v$-compatible.
But then Theorem~\ref{main-tech} implies that \eqref{dual} has an optimal solution that is dyadic as required.
\qed
\end{proof}
\noindent
In the remainder of the paper we shall prove Theorem~\ref{main-tech}.
\section{A lemma on binary clutters.}\label{sec-binary}
Here is the key result of the section,
\begin{proposition}\label{3-inter}
Let $\zC$ be a binary clutter and let $S\in\zC$.
Suppose that there exists $B\in b(\zC)$ where $|S\cap B|\geq 3$.
Then there exists $B'\in b(\zC)$ for which $|S\cap B'|=3$ and $S\cap B'\subseteq S\cap B$.
\end{proposition}
\noindent
Before we proceed to the proof observe that the binary assumption cannot be omitted, even if we replace $3$ with any constant bound.
Namely, we have,
\begin{remark}
For every $n\in\Z_+$ there exists a clutter $\zC$ with $|E(\zC)|=2n$ and $S\in\zC$, such that
\begin{itemize}
\item there exists $B\in b(\zC)$ with $|S \cap B|=n$, and
\item for every $B\in b(\zC)$, $|S\cap B|\in\{1,n\}$.
\end{itemize}
\end{remark}
\begin{myproof}
Let $\zB$ be the clutter with groundset $\{1,\ldots,n\}\cup\{1',\ldots,n'\}$ and members $\{1,\ldots,n\}$ and for all $i\in[n]$ members $\{i,i'\}$.
Let $\zC=b(\zB)$. Then $S=\{1,\ldots,n\}\in\zC$ and the result follows.
\end{myproof}
\noindent
If we are allowed to pick the set $S\in\zC$ then the condition that $\zC$ be binary is not needed.
\begin{proposition}
Let $\zC$ be a clutter. Suppose that some $S\in\zC$ and some $B\in b(\zC)$ satisfy $|S\cap B|>1$.
Then there exists $S'\in\zC$ and $B'\in b(\zC)$ with $|S'\cap B'|\in\{2,3\}$.
\end{proposition}
\begin{myproof}
If $\zC$ is binary then the result holds by Proposition \ref{3-inter}.
Otherwise, it is proved in \cite{Seymour76} that $\zC$ contains as a minor, say $\zD=\zC/I\setminus J$, one of: 
$\zP_4=\{\{1,2\},\{2,3\},\{3,4\}\}$ or $\Delta_s=\{\{1,\ldots,s\} \cup \{0,i\}: i\in[s]\}$ where $s\geq 2$.
Note that $\{2,3\}\in \zP_4\cap b(\zP_4)$ and $\{0,1\}\in \Delta_s\cap b(\Delta_s)$.
Hence, there exists $S'\in\zD$ and $B'\in b(\zD)$ with $|S'\cap B'|=2$.
Then for some $I'\subseteq I$ and $J'\subseteq J$ we have $S=S'\cup I'\in\zC$ and $B=B'\cup J'\in b(\zC)$.
In particular, $|S\cap B|=|S'\cap B'|=2$.
\end{myproof}
\subsection{A restatement in terms of matroids.}
To prove Proposition \ref{3-inter} we shall first restate the result in the language of binary matroids.
Given a matroid $M$ with a fixed element $\Omega$ we define,
\[
\port_\Omega(M)=\{C\setminus\Omega: \mbox{$C$ is a circuit of $M$ containing $\Omega$}\}.
\]
Binary clutters are exactly the clutters that can be expressed as ports of binary matroids. Namely,
\begin{proposition}[\cite{Seymour76}]\label{port-rep}
Let $\zC$ be a clutter. Then the following are equivalent,
\begin{itemize}
\item $\zC$ is a binary clutter,
\item for some binary matroid $M$ and $\Omega\in E(M)$ we have $\zC=\port_\Omega(M)$.
\end{itemize}
\end{proposition}
\noindent
In addition we have the following attractive relation between matroid duality and the blocking relation.
\begin{proposition}[\cite{novick95} (2.A.1)]\label{port-blocker}
Let $M$ be a binary matroid and let $\Omega\in E(M)$. Then,
\[
b\bigl(\port_\Omega(M)\bigr)=\port_\Omega(M^*).
\]
\end{proposition}
\noindent
Here $M^*$ denotes the dual of $M$.
Oxley studied when a matroid contains a {\em quad}, namely a set of 4 elements that are exactly the intersection of a circuit and a cocircuit \cite{oxley84}.
We prove the following related rooted result for binary matroids.
\begin{proposition}\label{4-inter}
Let $M$ be a binary matroid and let $\Omega\in E(M)$. Consider $C$ and $D$ where
\begin{enumerate}[\;\;a.]
\item $D$ is a cocircuit of $M$ containing $\Omega$, 
\item $C$ is a circuit of $M$ containing $\Omega$, and
\item $|C\cap D|\geq 4$.
\end{enumerate}
Then there exists a circuit $C'$ of $M$ with $\Omega\in C'$,
$|C'\cap D|=4$ and $C'\cap D\subseteq C\cap D$.
\end{proposition}
\noindent
We postpone the proof until Section \ref{sec-proof4}.
Let us show that this implies our result on binary clutters.
\begin{myproof}[Proposition \ref{3-inter}]
By Proposition~\ref{port-rep}, $\zC=\port_{\Omega}(M^*)$, 
for some binary matroid $M$. 
Then $S=D\setminus\{\Omega\}$ where $D$ is a cocircuit of $M$ with $\Omega \in D$.
By Proposition~\ref{port-blocker}, $B=C\setminus\{\Omega\}$ 
where $C$ is a circuit of $M$ such that $\Omega \in C$.
By Proposition~\ref{4-inter}, there exists a circuit $C'$ of $M$
with $\Omega\in C'$, $|C'\cap D| = 4$ and $C'\cap D\subseteq C\cap D$.
Then by Proposition~\ref{port-blocker},
$B' := C' \setminus \{\Omega\}$ is a member of $b(\mathcal{C})$
where
\[
|S\cap B'| = |(D\cap C') \setminus \{\Omega\}| = 3.
\]
Moreover, observe that $S\cap B'\subseteq S\cap B$. 
\end{myproof}
\subsection{The proof of Proposition \ref{4-inter}.}\label{sec-proof4}
We shall require a few definitions first. 
A {\em signed graph} is a pair $(G,\Sigma)$ where $\Sigma\subseteq E(G)$.
We say that $C\subseteq E(G)$ is a {\em cycle} of $G$ if $C$ induces an eulerian subgraph of $G$.
It is a {\em polygon} if it induces a connected subgraph where every vertex has degree two.
A cycle $C$ of $G$ is an {\em even cycle} of $(G,\Sigma)$ if $|C\cap\Sigma|$ is even.
A set of the form $\Gamma=\Sigma\triangle\delta_G(U)$ where $\delta_G(U)$ is a cut, is a {\em signature} of $(G,\Sigma)$.
Observe that $(G,\Sigma)$ and $(G,\Gamma)$ have the same set of even cycles.

We denote by $\cycle(G)$ the set of all cycles of $G$ and by $\ecycle(G,\Sigma)$ the set of all even cycles of $(G,\Sigma)$.
The set $\cycle(G)$ are exactly the cycles of the graphic matroid with representation $G$ and we identify $\cycle(G)$ with that matroid.
The set $\ecycle(G,\Sigma)$ are exactly the cycles of a binary matroid that is a single lift of a graphic matroid, this is known as an even-cycle matroid \cite{guenin23}.
We identify $\ecycle(G,\Sigma)$ with that matroid.
We use the following folklore result that says that we can represent as signed graphs, binary matroids that are graphic after contraction of a single element.
This follows readily, by considering the $0,1$ representation of the binary matroid.
\begin{remark}\label{lift}
Let $M$ be a binary matroid and let $e\in E(M)$ where $e$ is not a loop.
If $M/e=\cycle(G)$ then $M=\ecycle(G',\Sigma)$ where $G'$ is obtained from $G$ by adding loop $e$ and $\Sigma\subseteq E(G)$ with $e\in\Sigma$.
\end{remark}
\noindent
Finally, we will need a characterization of the cocycles of even-cycle matroids.
\begin{proposition}[\cite{pivotto11} Remark 2.1]\label{cocircuit}
Let $(G,\Sigma)$ be a signed graph.
The cocycle of $\ecycle(G,\Sigma)$ are the cuts of $G$ and the signatures of $(G,\Sigma)$.
\end{proposition}
We are now ready for the principal proof in this section.
\begin{myproof}[Proposition \ref{4-inter}]
We may assume $|C\cap D|\geq 6$.
It suffices to prove the following weaker statement,
\begin{equation}\tag{$\star$}
\begin{aligned}
&\text{There exists a circuit $C'$ with $\Omega\in C'$ where}\\
&4\leq |C'\cap D|<|C\cap D|
\text{ and } C'\cap D\subseteq C\cap D.
\end{aligned}
\end{equation}
Suppose for a contradiction that ($\star$) does not hold for $M$ with $C$ and $D$ satisfying (a)-(c).
We assume that $M$ among all such counterexamples minimizes $|E(M)|$.
\begin{myclaim}\label{inter-cl1}
Let $\hat{M}=M/I\setminus J$ where $\Omega\notin I\cup J\neq\emptyset$, $I\cap J=\emptyset$ and $(I \cup J) \cap C \cap D = \emptyset$.
Then, $\hat{C}=C\setminus I$ is not a circuit of $\hat{M}$ or $\hat{D}=D\setminus J$ is not a cocircuit of $\hat{M}$.
\end{myclaim}
\begin{proofofclaim}
Suppose for a contradiction that $\hat{C}$ is a circuit of $\hat{M}$ and that $\hat{D}$ is a cocircuit of $\hat{M}$.
As $\hat{C}\cap\hat{D}=C\cap D$, we have $\Omega\in \hat{C}\cap\hat{D}$, and $|\hat{C}\cap\hat{D}|\geq 6$.
It follows that (a)-(c) hold for $\hat{M}$ with $\hat{C},\hat{D}$.
By our minimality assumption, there exists a circuit $F$ of $\hat{M}$ with $\Omega\in F$ and $4\leq |F\cap\hat{D}|<|\hat{C}\cap\hat{D}|$.
By definition of minors, there exists $I'\subseteq I$ and $J'\subseteq J$ such that
$C'=F\cup I'$ is a circuit of $M$ and $D=\hat{D}\cup J'$.
Then $C'\cap D=(F\cup I')\cap (\hat{D}\cup J')=F\cap \hat{D}$.
It follows that $4\leq |F\cap\hat{D}|=|C'\cap D|$ and $|C\cap D|=|\hat{C}\cap\hat{D}|>|F\cap\hat{D}|=|C'\cap D|$.
But then ($\star$) holds for $M$, a contradiction.
\end{proofofclaim}
\begin{myclaim}\label{inter-cl2}
$C=D$ and all elements of $E(M)\setminus C$ are in the span of $C$.
\end{myclaim}
\begin{proofofclaim}
If there exists $e\in C\setminus D$ then let $\hat{M}=M/e$.
Then $C\setminus\{e\}$ is a circuit of $\hat{M}$ and $D$ is a cocircuit of $\hat{M}$, contradicting Claim \ref{inter-cl1}.
If there exists $e\in D\setminus C$ then let $M'=M\setminus e$.
Then $D\setminus\{e\}$ is a cocircuit of $\hat{M}$ and $C$ is a circuit of $\hat{M}$, contradicting Claim \ref{inter-cl1}.
Thus we have $C=D$.
If there exists $e\notin C=D$ where $e$ is not in the span of $C$, let $\hat{M}=M/e$.
Then $C$ is a circuit of $\hat{M}$ and $D$ is a cocircuit of $\hat{M}$, contradicting Claim \ref{inter-cl1}.
\end{proofofclaim}
Denote by $\Omega, e_1,\ldots,e_n$ the elements of $C=D$.
Note, $n\geq 5$ since $|C\cap D|\geq 6$.
Let $F=\{f_1,\ldots,f_m\}$ denote the elements in $E(M)\setminus C$. 
Pick $i\in[m]$.
By Claim~\ref{inter-cl2}, there are exactly two circuits $L,L'$ in $C\cup\{f_i\}$ and in addition $L\tr L'=C$.
Thus exactly one of $L, L'$ contains $\Omega$. Denote that circuit by $C_i$.
If $|C_i|\geq 5$ then $4\leq |C_i\cap D|<|C\cap D|$ and ($\star$) holds with $C'=C_i$.
Because circuits and cocircuits have even intersection in binary matroids, for every $i\in[m]$, $C_i$ is a triangle (a circuit of cardinality three) with elements say, $e_i,f_i,\Omega$.

Let $N=M/\Omega$.
Then $C\setminus\{\Omega\}$ is a circuit of $N$.
Moreover, for every $i\in[m]$, $e_i,f_i$ are parallel in $N$.
This implies that every element of $f_i$ is paired with some element $e_i$, in particular, $n\geq m$.
Let $N'$ denote the matroid obtained from $N$ by replacing every parallel class $e_i,f_i$ by $e_i$.
Then $N'$ consists of a single circuit $C\setminus\{\Omega\}$, in particular, it is graphic.
It follows that $N=\cycle(G)$ where $G$ is the graph 
obtained from the polygon $C\setminus\{\Omega\}$ by adding for each $i\in[m]$ an edge $f_i$ parallel to $e_i$.
($G$ is unique up to reordering $e_1,\ldots,e_n$.)

By Remark \ref{lift}, $M=\ecycle(H,\Sigma)$ 
where $H$ is obtained from $G$ by adding a loop $\Omega$ and for a suitable $\Sigma\subseteq E(H)$ with $\Omega\in\Sigma$.
Proposition \ref{cocircuit} implies that the cocircuits $M$ are either cuts of $H$ or signatures of $(H,\Sigma)$.
Since $D$ is a cocircuit of $M$ and $\Omega$ is a loop of $H$, $D$ must be a signature of $(H,\Sigma)$.
Since the signatures of $(H,\Sigma)$ and $(H,D)$ are the same, we will assume that $\Sigma=D$.
Suppose that $n\geq m+2$.
Then there exists $e_{m+1},e_{m+2}\in D$ which are not parallel to edges in $F$ in the graph $H$.
It follows that $D'=\{e_{m+1},e_{m+2}\}$ is a cut of $H$, hence, by Proposition \ref{cocircuit}, $D'$ is a cocycle of $M$.
But $D'\subset D$, contradicting the fact that $D$ is a cocircuit.
Thus $n\leq m+1$ and $m\geq n-1\geq 4$.
In particular, there exists $f_1,f_2\in F$ and for $i=1,2$, $e_i,f_i$ are parallel in $H$.
It follows that $C'=C\setminus\{e_1,e_2\}\cup\{f_1,f_2\}$ is a polygon of $H$.
Moreover, $|C'\cap\Sigma|$ is even.
It follows that $C'$ is a circuit of $M$ which satisfies ($\star$) as required.
\end{myproof}
\subsection{A geometric application.}
Given a clutter $\zC$ and $S\in\zC$, we denote by $\chi(S)$ the characteristic vector for $S$, i.e.
for every $e\in E(\zC)$, $\chi(S)_e=1$ if $e\in S$ and $\chi(S)_e=0$ otherwise.
In this section we describe a variant of Proposition \ref{3-inter} that describes where the point $\chi(B')$ can be selected on the polyhedron $P(\zC)$. 
Namely, the following holds.
\begin{proposition}\label{3-neighbour}
Let $\zC$ be a binary clutter, let $S\in\zC$, and let $\bar{z}$ be an integral corner of $P(\zC)$.
Suppose that
\begin{enumerate}[(i)]
\item $\bar{z}\in H_S:=\{x\in\R^n:\sum_{i\in S}x_i=1\}$, 
\item every ray of $P(\zC)$ incident to $\bar{z}$ is contained in $H_S$, and 
\item there exists $B\in b(C)$ such that $|B\cap S|>1$.
\end{enumerate}
Then, there exists $B'\in b(\zC)$ with $|S\cap B'|=3$ and where $\chi(B')$ is a vertex of $P(\zC)$ adjacent to $\bar{z}$. 
\end{proposition}
\noindent
Before we proceed with the proof recall the following result about tangent cones, see \cite{RW1998} Chapter 6, and also \cite{Ziegler1995}.
\begin{proposition}\label{cone}
Let $P$ be a polyhedron with a vertex $\bar{z}$. 
Let $L_1,\ldots,L_k$ be the faces of dimension $1$ that contain $\bar{z}$ and for all $i\in[k]$ let $v^i$ be a point in $L_i-\{\bar{z}\}$. 
Then $P\subseteq\bar{z}+\cone(v^i-\bar{z}:i\in[k])$. 
\end{proposition}
\begin{myproof}[Proposition \ref{3-neighbour}]
By hypothesis (iii) and Proposition \ref{3-inter} there exists $R\in b(\zC)$ with $|S\cap R|=3$. Then $\chi(R)\in P(\zC)$.
Let $L_1,\ldots,L_k$ denote the faces of dimension~$1$ of $P(\zC)$ that contain $\bar{z}$.
We may assume that for some $1\leq t \leq k$, $L_1,\ldots,L_t$ are bounded faces (i.e. line segments) and that $L_{t+1},\ldots,L_k$ are rays.
For $i\in[t]$, denote by $y^i$ the extreme point of $P(\zC)$ in $L_i\setminus\{\bar{z}\}$.
Observe that $y^i=\chi(B_i)$ for some $B_i\in b(\zC)$ (\cite{cornuejols01} Remark 1.16).
Since the recession cone of $P(\zC)$ is $\R_+^{E(\zC)}$ for $i\in[k]\setminus[t]$, $L_i$ is the extreme ray $\bar{z}+\R_+{\bf 1}_{\alpha(i)}$ for some index $\alpha(i)\in E(\zC)$.
It follows from Proposition~\ref{cone} that,
\begin{equation}\label{cone-eq1}
\chi(R)=\bar{z}+\sum_{i\in[t]} \lambda_i (y^i-\bar{z}) + \sum_{i\in[k]\setminus[t]} \lambda_i{\bf 1}_{\alpha(i)},
\end{equation}
where $\lambda_i\geq 0$ for all $i\in[k]\setminus[t]$.
By (i), $\chi(S)^\top\bar{z}=1$ and by (ii) $\alpha(i)\notin S$ for every $i\in[k]\setminus[t]$.
Left multiplying \eqref{cone-eq1} by $\chi(S)^\top$ then yields,
\begin{equation}\label{cone-eq2}
3=|S\cap R|=1+\sum_{i\in[t]}\lambda_i(|S\cap B_i|-1).
\end{equation}
In particular, there exists $\hat{\i}\in[t]$ for which $|S\cap B_{\hat\i}|\geq 2$.
We have $\bar{z}=\chi(B_0)$ for some $B_0\in b(\zC)$.
Because of \eqref{cone-eq1} we have that $R$ contains the support of,
\[
\bar{z}+\sum_{i\in[t]} \lambda_i (y^i-\bar{z}).
\]
In particular, $R\supseteq B_{\hat\i}\setminus B_0$. Hence,
\[
3=|S\cap R|\geq |S\cap (B_{\hat\i}\setminus B_0)|\geq |S\cap B_{\hat\i}|-|S\cap B_0|=|S\cap B_{\hat\i}|-1.
\]
Thus $|S\cap B_{\hat\i}|\leq 4$, $|S\cap B_{\hat\i}|\geq 2$ and since $\zC$ is binary we have $|S\cap B_{\hat\i}|=3$.
Then $B'=B_{\hat\i}$ satisfies the outcome of the proposition.
\end{myproof}
\section{The case with few active dual variables.}\label{sec-small}
We will show that Conjecture~\ref{corner-cn} holds when the number of active dual variables is small.
First we need the straightforward observation.
\begin{remark}\label{binary}
Let $\zC$ be a binary clutter that contains a non-empty set.
Then no set of $\zC$ is the symmetric difference of an even number of sets of $\zC$.
\end{remark}
We also require the following easy observation about integral corners.
\begin{remark}\label{corner-rem}
Let $\zC$ be a clutter and let $v$ be an integral corner of $P(\zC)$.
Let $L$ be a face of dimension $1$ containing the vertex $v$. Then either,
\begin{enumerate}[(i)]
\item $L$ is a line segment with an endpoint $\chi(B)\neq v$ such that $B\in b(\zC)$, or
\item $L$ is a half-line, $\{v+\lambda \1_j : \lambda\geq 0\}$ for some $j\in E(\zC)$.
\end{enumerate}
Moreover, if a face $F$ contains $v$, then $L\subset F$ if and only if $L\cap F\setminus\{v\}\neq\emptyset$.
\end{remark}
%
\noindent
Given a set $S\subseteq\R^n$, $\aff(S)$ defines the affine hull of $S$.
We are now ready for the principal result of this section.
\begin{proposition}\label{small}
Let $\zC$ be a binary clutter and let $w$ be minimal weights of $\zC$.
Let $F$ denote the $w$-optimal face of $P(\zC)$.
Suppose that $F$ contains an integral corner $v$ and suppose that $\sigma(F)\leq 6$.
Then \eqref{dual} has an optimal solution that is half-integer.
\end{proposition}
\begin{proof}
We say that minimal weights $w'$ of $\zC$ is a {\em counterexample} if 
(i) for the $w'$-optimal face $F'$ we have $\sigma(F')\leq 6$, 
(ii) $F'$ contains an integral corner $v$, and
(iii) \eqref{dual} does not admit an optimal solution that is half-integer.
A counterexample $w'$ is {\em extremal} if among all counterexamples it minimizes $\1^\top w'$.
Since the non-negative numbers are well-quasi ordered if there exists a counterexample, then there exists an extremal counterexample.
We will thus proceed by contradiction and assume that $w$ is an extremal counterexample.

We write $\sigma$ for $\sigma(F)$ throughout.
Denote by $\tau(w)$ the common optimal value of \eqref{primal} and \eqref{dual}.
Since $v\in F$ is a vertex, $v=\chi(B_v)$ for some $B_v\in b(\zC)$.
It follows that $\tau(w)$ is the weight of the $w$-minimum covers of $\zC$. 
Let $j\in E(\zC)$. Since $w$ is minimal, there exists a $w$-minimum cover $B$ of $\zC$ with $j\in B$.
Thus $\bar{x}:=\chi(B)\in F$. But $\bar{x}_j=1$, in particular, $x_j\geq 0$ is not tight for $F$.
It follows that the tight constraints for $F$ are of the form, $x(S)=1$ \footnote{$x(S):=\sum_{i\in S}x_i$.}
for every $S\in\zD$ where $\zD$ is some subset of $\zC$ with $|\zD|=\sigma$.
Define the following primal-dual pair obtained from \eqref{primal} and \eqref{dual} by restricting ourselves to the clutter $\zD$,
\begin{equation}
\min\{w^\top x:M(\zD)x\geq\1,x\geq\0\}, \label{primal-rest}
\end{equation}
\begin{equation}
\;\;\max\{\1^\top z: M(\zD)^\top z\leq w, z\geq\0\}. \label{dual-rest}
\end{equation}
Let $\bar{y}$ denote an optimal solution to \eqref{dual}.
By Complementary Slackness (CS), $\bar{y}_S=0$ for every $S\in\zC\setminus\zD$.
Define $\bar{z}$ where $\bar{z}_S=\bar{y}_S$ for every $S\in\zD$.
We say that $\bar{z}$ is the {\em restriction} of $\bar{y}$ to $\zD$ and that $\bar{y}$ is the {\em extension} of $\bar{z}$ to $\zC$.
Then $\bar{z}$ is feasible for \eqref{dual-rest}, in particular, the optimal value for \eqref{dual-rest} is $\geq$ the optimal value for \eqref{dual}.
Hence, the optimal value for \eqref{primal-rest} is $\geq$ the optimal value for \eqref{primal}.
But \eqref{primal-rest} is a relaxation of \eqref{primal}, so both LPs have the same common value $\tau(w)$.
In particular, every $w$-minimum cover of $\zC$ is a $w$-minimum cover of $\zD$ (but the converse need not hold).
When $w$ is clear from context, we refer to $w$-minimum covers simply as minimum covers.
A {\em transversal} of a clutter is a subset of the groundset that intersects every member of the clutter exactly once.
Throughout this proof, $A$ denotes the matrix $M(\zD)$. 
Denote by $r^1,\ldots,r^\sigma$ the rows of $A$ where $r^i$ is the characteristic vector of member $S_i\in\zD$.
\begin{myclaim}\label{cl-dual}
\begin{enumerate}[(a)]
\item For every $j\in E(\zC)$ there exists a minimum cover $B$ of $\zD$ with $j\in B$.						
\item Every minimum cover $B$ of $\zD$ is a transversal of $\zD$.										
\item Let $S\in\zC$. If for every minimum cover $B$ of $\zD$, $|S\cap B|=1$ and $\forall j\in E(C): \ w_j=0$ we have $j\notin S$, then $S\in\zD$.					
\item For every optimal solution $\bar{x}$ of \eqref{primal-rest} we have $M(\zD)\bar{x}=\1$.					
\item For every optimal solution $\bar{z}$ of \eqref{dual-rest} we have $M(\zD)^\top\bar{z}=w$.					
\item There is no $S\in\zD$ for which $\bar{z}_S=0$ for every optimal solution $\bar{z}$ of \eqref{dual-rest}.		
\item For every $j\in E(\zC)$, $\col_j(A)=\0$ if and only if $w_j=0$.										
\end{enumerate}
\end{myclaim}
\vspace{-0.1in}
\begin{proofofclaim}
{\bf (f)}
By strict CS, there exists an optimal solution $\bar{y}$ of \eqref{dual} where $\bar{y}_S>0$ for every $S\in\zD$.
Then the restriction of $\bar{y}$ to $\zD$ is an optimal solution to \eqref{dual-rest}.
{\bf (a)}
Follows from the fact that \eqref{primal} and \eqref{primal-rest} have the same optimal value $\tau(w)$ and the fact that $w$ is minimal.
{\bf (d)}
Follows from (f) and CS.
{\bf (b)}
If $B$ is a minimum cover, then $\chi(B)$ is an optimal solution to \eqref{primal-rest}. The result now follows from (d).
{\bf (e)}
Follows from (a) and CS.
{\bf (c)}
Let $L_1,\ldots,L_k$ denote the faces of dimension~$1$ containing $v$ that are included in $F$. 
Observe that $\spn(L_1\cup\ldots\cup L_k)=\aff(F)$.
To show that $S\in\zD$ we need to show that $x(S)\geq 1$ is tight for every $\bar{x}\in F$.
It suffices to show that $x(S)\geq 1$ is tight for every $L=L_i$, $i\in[k]$.
We consider separately the two outcomes of Remark~\ref{corner-rem}.
First, suppose that $L$ is a line segment, as in outcome~(i).
Then $L$ is a line segment with ends $v=\chi(B_0)$ and $\chi(B)$ where $B_0,B_i\in b(\zC)$.
By hypothesis $x(S)\geq 1$ is tight for $\chi(B_0)$ and $\chi(B)$, hence by Remark~\ref{corner-rem}, for $L$ as required.
Consider outcome (ii) in Remark~\ref{corner-rem} for $L$.
Then $L=\{v+\lambda\1_j:\lambda\geq 0\}$ for some $j\in E(\zC)$.
Since $w$ is minimal, we must have $w_j=0$, therefore, $j\notin S$. Therefore,  $v$ and $v+\1_j$ are tight for $x(S)\geq 1$, which directly results that it is tight for $L$ as required.
{\bf (g)}
Suppose $\col_j(A)=\0$.
By (a) there exists a cover $B\ni j$ of $\zD$ with $w(B)=\tau(w)$.
Because $\col_j(A)=\0$, $B':=B\setminus\{j\}$ is a cover of $\zD$.
Hence, $\tau(w)\leq w(B')=w(B)-w_j=\tau(w)-w_j$, i.e. $w_j\leq 0$.
Suppose $w_j=0$ and suppose that there exists $S\in\zD$ with $j\in S$.
Then for every feasible solution $\bar{z}$ of \eqref{dual-rest} we have $\bar{z}_S=0$, contradicting (f).
\end{proofofclaim}
\begin{myclaim}\label{cl-dual1}
Let $\bar{z}$ be an optimal solution to \eqref{dual-rest}. Then $\bar{z}_S<1$ for every $S\in\zD$.
\end{myclaim}
\begin{proofofclaim}
Suppose for a contradiction that we have an optimal solution $\bar{z}$ of \eqref{dual-rest} with $\bar{z}_R\geq 1$ for some $R\in\zD$.
Let $w'=w-\chi(R)\geq\0$.
Define, $\bar{y}$ where
\[
\bar{y}_S=
\begin{cases}
\bar{z}_S -1 & \text{if $S=R$} \\
\bar{z}_S & \text{if $S\in\zD$ and $S\neq R$} \\
0 & \text{if $S\in \zC\setminus\zD$.}
\end{cases}
\]
Note that $\bar{y}$ is a feasible solution for \eqref{dual} for weights $w'$ of value $\tau(w)-1$.
Let $\bar{x}\in F$.
By CS, $\bar{x}(R)=1$.
Then $w'(\bar{x})=w(\bar{x})-\bar{x}(R)=\tau(w)-1$.
It follows that $\bar{x}$ and $\bar{y}$ form a pair of optimal solutions of \eqref{primal} and \eqref{dual} for weights $w'$.
As $\bar{x}$ was arbitrary, $F\subseteq F'$ where $F'$ is the $w'$-optimal face.
Therefore, $\sigma(F')\leq\sigma(F)\leq 6$ and $w'$ is minimal.
Since $w$ is an extremal counterexample and $\1^\top w'<\1^\top w$,
\eqref{dual} has a half-integer optimal dual solution $y'$ for weights $w'$.
Then let $y''$ where $y''_R=y'_R+1$ and $y''_S=y'_S$ for every $S\in\zC \setminus \{R\}$.
Thus, $y''$ is a half-integral feasible solution to \eqref{dual} with objective value $\tau(w)$, and is therefore optimal, a contradiction.
\end{proofofclaim}
\noindent
A column $j$ of a matrix $N$ is {\em implied} if either (i) $\col_j(N)=\0$, (ii)  $\col_j(N)=\col_k(N)$ for some $k\neq j$, or (iii) $\col_j(N)$ is the sum of other columns of $N$.
Let $\hat{A}$ be matrix obtained from $A$ by repeatedly removing any implied column.
Denote by $\hat{w}$ the restriction of $w$ to columns of $\hat{A}$.
For $i\in[\sigma]$ denote by $\hat{r}^i$ the restriction of $r^i$ to $\hat{A}$ and denote by $\hat{S}_i$ the restriction of $S_i$ to the columns indices of $\hat{A}$.
\begin{myclaim}\label{cl-hat}
\begin{enumerate}[(a)]
\item $\hat{A}$ has a column with some entry $0$ and some entry $1$.
\item $A^\top z=w$ and $\hat{A}^\top z=\hat{w}$ have the same solutions.
\item if $\hat{A}^\top\bar{z}=\hat{w}$ and  $\bar{z}\geq\0$  then the extension of $\bar{z}$ to $\zC$ is an optimal solution to~\eqref{dual}.
\end{enumerate}
\end{myclaim}
\begin{proofofclaim}
{\bf (a)}
Otherwise $A=\0$, which implies $\emptyset\in\zC$, a contradiction.
{\bf (b)}
If $\bar{z}$ is a solution to $A^\top z=w$ then it is a solution to $\bar{A}^\top z=\hat{w}$ as the latter is a row subsystem of the former.
If $\bar{z}$ is a solution to $\hat{A}^\top z=\hat{w}$ then it is a solution to $A^\top z=w$ since, 
(i) every column in $A$ is a linear combination of columns in $\hat{A}$, and
(ii) $A^\top z=w$ has at least one solution because of Claim~\ref{cl-dual}(e).
{\bf (c)}
Let $\bar{x}$ be an optimal solution to \eqref{primal-rest}
Since $\bar{x}$ and $\bar{z}$ satisfy the CS conditions for \eqref{primal-rest} and \eqref{dual-rest}, $\bar{z}$ is an optimal solution to \eqref{dual-rest}.
As both \eqref{dual} and \eqref{dual-rest} have the same optimal value $\tau(w)$, the result follows.
\end{proofofclaim}
\begin{myclaim}\label{cl-col}
Let $j$ be the index of a column of $\hat{A}$. Then
\begin{enumerate}[(a)]
\item $\hat{w}_j\leq\1^\top\col_j(\hat{A})-1$, and 
\item $2\leq\1^\top\col_j(\hat{A})\leq\sigma-2$.
\end{enumerate}
\end{myclaim}
\begin{proofofclaim}
{\bf (a)}
Let $\bar{z}$ be any optimal solution to \eqref{dual-rest}.
Claim \ref{cl-dual1} implies that $\bar{z}_S<1$ for every $S\in\zD$. 
Then $w_j=\bar{z}^\top\col_j(\hat{A})<\1^\top\col_j(\hat{A})$ where the equality follows from Claim \ref{cl-dual}(e).
{\bf (b)}
Suppose $\1^\top\col_j(\hat{A})\leq 1$. 
Then by definition of $\hat{A}$, $\1^\top\col_j(\hat{A})=1$.
(a) implies $w_j=0$.
It follows then from Claim \ref{cl-dual}(g) that $\col_j(\hat{A})=\0$, a contradiction.
Suppose $\1^\top\col_j(\hat{A})\geq \sigma-1$.
Then by definition of $\hat{A}$, $\1^\top\col_j(\hat{A})=\sigma-1$.
But then by Claim \ref{cl-dual}(a), (b) we have some column $k$ of $\hat{A}$ for which $\1^\top\col_k(\hat{A})=1$, a contradiction.
\end{proofofclaim}
\noindent
Claims~\ref{cl-hat}(a) and \ref{cl-col}(b) imply that $\sigma\geq 4$.
Moreover, if $\sigma=4$ then every column of $\hat{A}$ has exactly two $1$s and two $0$s.
In that case, by Claim~\ref{cl-col}(a) and Claim \ref{cl-dual}(g), $\hat{w}=\1$.
Then $\bar{z}=\frac12\1$ is a solution to $\hat{A}z=\hat{w}$.
Claim~\ref{cl-hat}(c) then implies that \eqref{dual} has a half-integer solution, a contradiction.
Thus $\sigma\geq 5$.
\begin{myclaim}\label{cl-rank}
We have $r(\hat{A})\geq 4$ where $r(\cdot)$ denote the rank of a matrix over $\mathbb{F}_2$.
\end{myclaim}
\begin{proofofclaim}
Suppose for a contradiction, that $r(\hat{A})\leq 3$. Then for some $m\leq 3$ we have rows $r^1,\ldots,r^m$ of $A$ that span every row of $A$ over $\mathbb{F}_2$.
It follows from Remark~\ref{binary} that if $m\leq 2$ then $\sigma=m$ and that if $m=3$ then $\sigma\leq 4$.
But we observed that $\sigma\geq 5$, a contradiction.
\end{proofofclaim}
\begin{myclaim}\label{cl-2cover}
Consider sets $S_1,\ldots,S_{2t+1}\in\zD$ for some $t\in\Z_+$, and let $S$ be the symmetric difference of all of $S_i$, $i\in[2t+1]$.
Then for every minimum cover $B$ with $|B|=2$ we have $|B\cap S|=1$.
\end{myclaim}
\begin{proofofclaim}
Denote by $p,q$ the elements of $B$.
Let $A'$ be the row submatrix of $A$ with rows $r^1,\ldots,r^{2t+1}$.
Claim~\ref{cl-dual}(b) implies that $\1^\top\col_j(A')+\1^\top\col_k(A')=2t+1$.
Thus we may assume that $\1^\top\col_p(A')$ is odd and $\1^\top\col_q(A')$ is even.
It then follows that $p\in S$ and $q\notin S$.
\end{proofofclaim}
\begin{myclaim}\label{cl-3cover}
There exists a minimum cover $\{j,k,\ell\}$ with $\1^\top\col_j(\hat{A})=\1^\top\col_k(\hat{A})=\1^\top\col_\ell(\hat{A})=2$, and $\hat{w}_j=\hat{w}_k=\hat{w}_\ell=1$.
\end{myclaim}
\begin{proofofclaim}
Suppose for a contradiction that every minimum cover has cardinality at most two.
Since $r(A)\geq 4$ there exists rows $r^1,r^2,r^3,r^4$ of $A$ that are independent over $\mathbb{F}_2$.
For every $i\in[4]$ let $R_i$ be the symmetric difference of $S_t$ over all $t\in[4]\setminus\{i\}$.
Since $\zC$ is binary and by Claim~\ref{cl-dual}(a), $R_i\in\zC$.
By Claim~\ref{cl-2cover} and the hypothesis, $R_i$ is a transversal of the minimum covers of $\zD$.
Claim~\ref{cl-dual}(c), $R_i\in\zD$.
Rows of $A$ corresponding to $S_1,\ldots,S_4,R_1,\ldots,R_4$ are linearly independent over $\mathbb{F}_2$.
Hence, $\sigma\geq 8$, a contradiction.
Thus a minimum cover of cardinality three exists.
For the rest, Claim \ref{cl-col}(b) and $\sigma \geq 5$ implies each column $\{j,k,\ell\}$ has at least two ones and two zeros. Then Claim~\ref{cl-dual}(b) implies each tight constraint includes exactly one element $\{j,k,\ell\}$. Together with assumption $\sigma\leq 6$,  $\1^\top\col_j(\hat{A})=\1^\top\col_k(\hat{A})=\1^\top\col_\ell(\hat{A})=2$. Moreover, Claim~\ref{cl-dual}(a) implies $\hat{w}_j=\hat{w}_k= \hat{w}_{\ell}=1$.
\end{proofofclaim}
\noindent
The previous result implies immediately that $\tau(w)=3$.
\begin{myclaim}\label{cl-mixed}
There exists a minimum cover $\{p,q\}$ with $\1^\top\col_p(A)=\1^\top\col_q(A)=3$.
\end{myclaim}
\begin{proofofclaim}
Suppose first no minimum cover of cardinality two exists.
Then $\hat{w}=\1$ by Claim~\ref{cl-3cover}.
It follows that $\bar{z}=\frac12\1$ is a solution to $\hat{A}z=\hat{w}$.
Claim~\ref{cl-hat}(c) then implies that \eqref{dual} has a half-integer solution, a contradiction.
Thus there exists a minimum cover of cardinality two.
Suppose now that for every minimum cover $\{p,q\}$ we have $\1^\top\col_p(\hat{A})=4$ and $\1^\top\col_q(\hat{A})=2$.
By Claim~\ref{cl-col}(a), $\hat{w}_q=1$, thus $\hat{w}_p=2$.
Again, the extension of $\bar{z}=\frac12\1$ to $\zC$ is a half-integer solution to \eqref{dual}.
\end{proofofclaim}
\noindent
Because of Claim~\ref{cl-rank} one of Case 1 or Case 2 must hold.

\begin{description}
\item {\bf Case 1:} $r(A)\geq 5$.
\end{description}

Let $r^1,\ldots,r^5$ denote rows of $A$ that are independent over $\mathbb{F}_2$.
Denote by $S$ the symmetric difference of $S_1,\ldots,S_5$.
Let $B=\{j,k,\ell\}$ denote an arbitrary minimum cover of cardinality three.
We may assume that each of $\col_j(A)$ and $\col_k(A)$ has exactly two $1$s in rows $r^1,\ldots,r^5$ and that $\col_\ell(A)$ has exactly one $1$ in $r^1,\ldots,r^5$.
It then follows that $j,k\notin S$ and $\ell\in S$.
Hence, $B\cap S=\{\ell\}$.
Together with Claim \ref{cl-2cover}, it follows that $S$ is a transversal of the minimum covers of $\zD$.
Therefore, by Claim~\ref{cl-dual}(c), $S$ is the (unique) member of $\zD$ distinct from $S_1,\ldots,S_5$.
By Claim \ref{cl-mixed} there exists a minimum cover $\{p,q\}$ with $\1^\top\col_p(A)=\1^\top\col_q(A)=3$.
We may assume that $\col_p(A)$ has exactly three $1$s in rows $r^1,\ldots,r^5$.
But then $p\in S$, which implies $\1^\top\col_p(A)=4$, a contradiction. 

\begin{description}
\item {\bf Case 2:} $r(A)=4$.
\end{description}

A cover $B$ of $\zD$ is {\em illegal} if $w(B)\leq\tau(w)$ and $B$ is not a transversal of $\zD$.
Claim~\ref{cl-dual}(b) implies that no cover is illegal. 
Denote the row indices of $\hat{A}$ by $1,\ldots,6$ and the columns of $\hat{A}$ by $1,\ldots,n$ for some $n\in\Z_+$.
We may assume (after row and column permutations) by Claim \ref{cl-3cover} that $\{1,2,3\}$ is a minimum cover of $\zD$ and that 
$\col_1(\hat{A})$ has $1$s in rows $1,4$; 
$\col_2(\hat{A})$ has $1$s in rows $2,5$; and 
$\col_3(\hat{A})$ has $1$s in rows $3,5$.
We may assume by Claim \ref{cl-mixed} that $\{4,5\}$ is a minimum cover where $\1^\top\col_4(\hat{A})=\1^\top\col_5(\hat{A})=3$.
Consider first the case where one of columns $1,2,3$ have both $1$s in their column contained in the $1$s of column $4$ or $5$.
Then we may assume, that 
$\col_4(\hat{A})$ has $1$s in rows $1,2,4$, and 
$\col_5(\hat{A})$ has $1$s in rows $3,5,6$.
We have $\hat{w}_1=\hat{w}_2=\hat{w}_3=1$ and we may still assume $\hat{w}_4=1$, $\hat{w}_5=2$.
But then $\{2,3,4\}$ is an illegal cover.
Therefore, we may assume that,
$\col_4(\hat{A})$ has $1$s in rows $1,2,3$, and
$\col_5(\hat{A})$ has $1$s in rows $4,5,6$.
Observe that columns $1,2,3,4,5$ of $\hat{A}$ has rank $4$ over ${\mathbb F}_2$.
Consider the matrix that contains all columns $1-5$ of $\hat{A}$ as well as all non-implied columns that are in the span of columns $1-5$ of $\hat{A}$,
\[
A^+=
\begin{array}{c}
\tiny{
\begin{array}{ccccccccccc}
1 & \;2 & \;3 & \;4 & \;5 & 6 & \;7 & \;8 & 9 & 10 & 11\;
\end{array}}
\\
\left(
\begin{array}{ccc|cc|ccc|ccc}
1 & 0 & 0 & 1 & 0 & 0 & 1 & 1 & 1 & 0 & 0 \\
0 & 1 & 0 & 1 & 0 & 1 & 0 & 1 & 0 & 1 & 0 \\ 
0 & 0 & 1 & 1 & 0 & 1 & 1 & 0 & 0 & 0 & 1 \\ \hline
1 & 0 & 0 & 0 & 1 & 1 & 0 & 0 & 0 & 1 & 1 \\ 
0 & 1 & 0 & 0 & 1 & 0 & 1 & 0 & 1 & 0 & 1 \\ 
0 & 0 & 1 & 0 & 1 & 0 & 0 & 1 & 1 & 1 & 0  
\end{array}
\right).
\end{array}
\]
Observe that $\hat{A}$ is a column submatrix of $A^+$.
For each column $j$ of $A^+$ associate a weight $w^+_j$ that would be the weight $w_j$ if $j$ were a column of $\hat{A}$.
We already established that $w^+_1=w^+_2=w^+_3=w^+_4=1$ and $w^+_5=2$.
$\{1,4,9\}$ is a cover, and it is illegal if $w^+_9=1$, thus $w^+_9=2$.
Since $\{6,9\}$ is a minimum cover, $w^+_6=1$.
Similarly, $w^+_7=w^+_8=1$ and $w^+_{10}=w^+_{11}=2$.
Note, both columns $6$ and $7$ cannot be present in $\hat{A}$ for otherwise, $\{3,6,7\}$ would be an illegal cover.
Therefore, we may assume that $\hat{A}$ is contained in columns $1-5$ and $6,9$ of $A^+$.
It follows that we may assume that $\hat{A}$ is contained in columns $1-5$ and $6,9$ of $A^+$.
But then,
$
\bar{z}=(\frac12,0,\frac12,\frac12,1,\frac12)^\top
$
is a solution to $\hat{A}z=\hat{w}$.
Claim~\ref{cl-hat}(c) then implies that \eqref{dual} has a half-integer solution, a contradiction.
\qed
\end{proof}
\section{The proof of Theorem \ref{main-tech}.}\label{sec-proof}
\subsection{A general tool.}
We shall state a result from \cite{guenin25b} that gives a sufficient condition for a linear program to have a dyadic solution.
We require a number of definitions.
Consider $A\in\Z^{m\times n}$ and $b\in\Z^m$ and let $P$ be the polyhedron $\{x:Ax\geq b\}$.
Let $F$ be a face of $P$.
Denote by $I(F)$ the set of row indices of $A$ corresponding to the tight constraints for $F$, i.e. $i\in I(F)$ if and only if $\row_i(A)x=b_i$ for every $x\in F$. 
Face $F$ is {\em compliant} if for every $w\in\Z^n$ for which $F$ is contained in  the $w$-optimal face, the following linear program has an optimal solution that is dyadic,
\begin{equation}\label{eq-tool}
\max\{b^\top y:A^\top y=w, y\geq\0\}.
\end{equation}
A sequence of faces of $P$, $(F_0,F_1,\ldots,F_k)$ is a {\em face-chain rooted at $F_0$} if for every $\ell\in[k]$, $F_\ell$ is an immediate superface of $F_{\ell-1}$.
This face-chain is {\em well-marked} if $F_k$ is compliant and 
for every $\ell\in[k]$ there exists a vector $\rho\in\Z^n$ and an index $\hat{\i}\in I(F_{\ell-1})\setminus I(F_{\ell})$ such that, \footnote{Note that $\rho$ and $\hat\i$ depend on $\ell$.}
\begin{enumerate}[\;\;(g1)]
\item $\rho\in\aff(F_\ell)\setminus\aff(F_{\ell-1})$ and 
\item $\row_{\hat\i}(A)\rho-b_{\hat\i}=\pm2^p$ for some $p\in\Z_+$.
\end{enumerate}
Here is our key tool from \cite{guenin25b}, it is a combination of Theorem 2.2 and Proposition 3.1 in a different language but without essential changes to the proof.
%
\begin{theorem}\label{tool}
Let $P=\{x:Ax\geq b\}\subset\R^n$ be a polyhedron where all entries of $A,b$ are integer.
Let $\bar{w}\in\Z^n$ and let $\bar{F}$ denote the $\bar{w}$-optimal face of $P$ where $\bar{F}\cap\Z^n\neq\emptyset$.
If there exists a well-marked face-chain rooted at $\bar{F}$, then there exists an optimal solution of \eqref{eq-tool} for $w=\bar{w}$ that is dyadic.
\end{theorem}
\noindent
We will work with the following primal dual pair,
\begin{equation}
\min\{w^\top x:Ax\geq b\}, \quad\quad\quad\;\;\label{gPrimal}
\end{equation}
\begin{equation}
\max\{b^\top y:A^\top y=w,y\geq\0\}.\label{gDual}
\end{equation}
\noindent 
Denote by $\D$ the dyadic rationals. We require the following (\cite{DGSS}, Lemma 2.2).
\begin{theorem}\label{density}
Let $P\subseteq\R^n$ be a rational polytope. Then $P\cap\D^n\neq\emptyset$ if and only if $\aff(P)\cap\D^n\neq\emptyset$.
\end{theorem}
\noindent 
We will leverage this to obtain a sufficient condition for \eqref{gDual} to have a feasible dyadic solution.
Consider the polyhedron $P=\{x:Ax\geq b\}$ and let $F$ be a face of $P$.
Assume $A$ has $m$ rows. 
Recall that $I(F)\subseteq[m]$ denotes the subset of the row indices of $A$ corresponding to the tight constraints for $F$.
\begin{proposition}\label{gCS}
Let $P=\{x:Ax\geq b\}$ be a rational polyhedron where $A$ is an $m\times n$ matrix.
Consider $w\in\Z^n$ and let $F\neq\emptyset$ denote the $w$-optimal face of $P$.
If there exists $\bar{y}\in\D^m$ where, $A^\top\bar{y}=w$ and $\bar{y}_i=0$ for every $i\in[m]\setminus I(F)$, then \eqref{gDual} has an optimal dyadic solution.
\end{proposition}
\begin{proof}
The set of optimal solutions to \eqref{gDual} is given by,
\[
\zS=\{y\geq\0:A^\top y=w, y_i=0, \;\forall i\in[m]\setminus I(F)\}.
\]
Indeed, if $\bar{x}\in F$ then $\bar{x}$ and $\bar{y}$ satisfy CS conditions
for \eqref{gPrimal}, \eqref{gDual} if and only if $\bar{y}\in S$.
By Strict CS, there exists $\bar{y}\in \zS$ where $\bar{y}_i>0$ for every $i\in I(F)$.
Since the affine hull of a polyhedron is described by the solutions to the tight constraints,
\[
\aff(\zS)=\{y:A^\top y=w, y_i=0, \;\forall i\in[m]\setminus I(F)\}.
\]
Assume we have $\bar{y}\in\D^m$ as in the hypothesis.
Then $\bar{y}\in\aff(\zS)$ and by Theorem~\ref{density} there exists $y'\in \zS\cap\D^m$, i.e. \eqref{gDual} has an optimal dyadic solution.
\qed
\end{proof}
\noindent
Consider a set $S$ of vectors $a^1,\ldots,a^k\in\Z^m$.
Then $S$ is said to be a {\em Dyadic Generating Set for Subspaces} (DGSS) if
every integer point that is a linear combination of vectors in $S$ can be expressed as a linear dyadic combination of vectors in $S$.
Similarly, $S$ is said to be a {\em Dyadic Generating Set for Cones} (DGSC) if
every integer point that is a conic combination of vectors in $S$ can be expressed as a conic dyadic combination of vectors in $S$.
We require (\cite{DGSS}, Proposition 3.3),
\begin{proposition}\label{C2S}
Every DGSC is a DGSS.
\end{proposition}
Note that the converse of Proposition~\ref{C2S} result does not hold. 
\begin{proof}[Theorem~\ref{tool}]
\setcounter{claimcounter}{0}
By hypothesis, there exists a well-marked face chain, say $(F_0,\ldots,F_k)$ where $F_0=\bar{F}$.
For every $\ell\in[k]$ define,
\[
\zA_{\ell}=\{row_i(A):i\in I(F_\ell)\},
\]
i.e. $\zA_\ell$ is the set of left-hand-sides of the tight constraints defining face $F_\ell$.
Since $F\cap\Z^n\neq\emptyset$, there exists an integer vector $v\in \bar{F}\subseteq F_0,\ldots,F_k$.
We write $\ker(\zA_\ell)$ for the vectors normal to every member of $\zA_\ell$.
For a vector space $\zS$ the orthogonal complement is denoted, ${\mathcal S}^\perp$.
\begin{myclaim}\label{cl-fco}
$\spn(\zA_\ell)^\perp=\nul(\zA_\ell)=\aff(F_\ell)-v$.
\end{myclaim}
\begin{proofofclaim}
Since $\aff(F_\ell)$ is the set of solutions to the tight constraints of $Ax\geq b$ for $F_\ell$, we have $\aff(F_\ell)-v=\nul(\zA_\ell)$.
The first equality follows from the Fundamental Theorem of Linear Algebra.
\end{proofofclaim}
\begin{myclaim}\label{cl-dgss}
$\zA_k$ is a DGSS.
\end{myclaim}
\begin{proofofclaim}
Since $(F_0,\ldots,F_k)$ is well-marked, $F_k$ is compliant.
Pick an arbitrary $w\in\Z^n$ that is a conic combinations of vectors in $\zA_k$.
Then observe that \eqref{gDual} with right-hand-side $w$ has an optimal solution.
Since $F_k$ is compliant, \eqref{gDual} has an optimal solution $\bar{y}\in\D^m$.
By CS, we have $\bar{y}_i=0$ for every $i\in[m]\setminus I(F_k)$.
But then $w=\sum_{i\in I(F_k)}\bar{y}_i\row_i(A)^\top$ where $\bar{y}\geq\0$.
Hence, $w$ can be expressed as a conic dyadic combinations of $\zA_k$.
Thus by definition, $\zA_k$ is a DGSC, and hence, a DGSS by Proposition \ref{C2S}.
\end{proofofclaim}
\begin{myclaim}\label{cl-done}
It suffices to prove that $\zA_0$ is a DGSS.
\end{myclaim}
\begin{proofofclaim}
Let $y'$ be a solution to \eqref{gDual} with right-hand-side $\bar{w}$.
By CS, $y'_i=0$ for every $i\in[m]\setminus I(F_0)$.
Since $\bar{w}=A^\top y'$, it follows that $\bar{w}$ is a linear combination of vectors in $\zA_0$.
If $\zA_0$ is a DGSS, then there exists $\bar{z}$, dyadic, for which
$\bar{w}=\sum_{i\in I(F_0)}\bar{z}_i\row_i(A)^\top$.
For every $i\in I(F_0)$ set $\bar{y}_i=\bar{z}_i$ and for every $i\in[m]\setminus I(F_0)$, set $\bar{y}_i=0$.
Then $\bar{y}$ satisfies the conditions in Proposition \ref{gCS},
thus \eqref{gDual} with right-hand-side $\bar{w}$, has an optimal dyadic solution.
\end{proofofclaim}
\begin{myclaim}\label{onestep}
It suffices to show for all $\ell\in[k]$ that if $\zA_\ell$ is a DGSS, then so is~$\zA_{\ell-1}$.
\end{myclaim}
\begin{proofofclaim}
By Claim \ref{cl-dgss}, $\zA_k$ is a DGSS.
If the statement in the claim holds, then by repeated applications we deduce that $\zA_0$ is a DGSS.
But then we are done by Claim \ref{cl-done}.
\end{proofofclaim}
\noindent
Thus we pick $\ell\in[k]$ for which $\zA_\ell$ is a DGSS.
Our goal will be to show that $\zA_{\ell-1}$ is also a DGSS.
To that effect pick an arbitrary $w^0\in\Z^n$ that is a linear combination of vectors in $\zA_{\ell-1}$.
It will suffice to show that $w^0$ can be expressed as a dyadic combination of vectors in $\zA_{\ell-1}$.
Since the face-chain $(F_0,\ldots,F_k)$ is well-marked, there exists $\rho\in\Z^n$ and $\hat\i\in I(F_{\ell-1}) \setminus I(F_{\ell})$ such that,
\begin{enumerate}[\;\;(g1)]
\item $\rho\in\aff(F_\ell)\setminus\aff(F_{\ell-1})$ and 
\item $\row_{\hat\i}(A)\rho-b_{\hat\i}=\pm2^p$ for some $p\in\Z_+$.
\end{enumerate}
Let $r=\rho-v$ and define,
\begin{equation}
w^1=w^0-u\row_{\hat\i}(A)^\top\quad\mbox{where}\quad u=\frac{r^\top w^0}{\row_{\hat\i}(A)r}. \label{w0w1}
\end{equation}
\begin{myclaim}\label{cl-span}
$w^1\in\spn(\zA_\ell)$.
\end{myclaim}
\begin{proofofclaim}
Since $w^0\in\spn(\zA_{\ell-1})$ and $\row_{\hat\i}(A)\in\zA_{\ell-1}$ we have $\spn(w^1)\subseteq\spn(\zA_{\ell-1})$.
Claim \ref{cl-fco} implies, $\nul(w^1)=\spn(w^1)^\perp\supseteq\spn(\zA_{\ell-1})^\perp=\nul(\zA_{\ell-1})$.
Moreover, $r\in\nul(w^1)$ since \eqref{w0w1} implies immediately that $r^\top w^1=0$.
Hence,
\begin{equation} \label{ker-eq}
\nul(w^1)\supseteq\nul(\zA_{\ell-1})+\spn\{r\}.
\end{equation}
Claim \ref{cl-fco} and (g1)  implies (a) $r\in\nul(\zA_{\ell})\setminus\nul(\zA_{\ell-1})$.
Claim \ref{cl-fco} also implies (b) $\dim(\nul(\zA_{\ell}))=\dim(F_\ell)=\dim(F_{\ell-1})+1=\dim(\nul(\zA_{\ell-1}))+1$.
(a) and (b) together with \eqref{ker-eq} implies that
$\nul(w^1)\supseteq\nul(\zA_\ell)$. 
Using Claim \ref{cl-fco} again yields, 
\[
\spn(w^1)=\nul(w^1)^\perp\subseteq\nul(\zA_\ell)^\perp=\spn(\zA_\ell), 
\]
as required
\end{proofofclaim}
\begin{myclaim}\label{cl-u}
$u$ is dyadic.
\end{myclaim}
\begin{proofofclaim}
Since $v,\rho\in\Z^n$ so is $r=\rho-v$.
Because, $w^0\in\Z^n$ we have $r^\top w^0\in\Z$, i.e. the nominator of $u$ in \eqref{w0w1} is integer.
Since $v\in F_\ell$ we have $\row_{\hat\i}(A)v=b_{\hat\i}$.
It follows from (g2) that $\row_{\hat\i}(A)r=\pm 2^p$ for some $p\in\Z_+$, i.e. the denominator of $u$ in \eqref{w0w1} is a power of two.
\end{proofofclaim}
\noindent
By Claim \ref{cl-u}, $w^1$ is dyadic.
Hence, $w^1=\frac1{2^p}w^2$ where $w^2\in\Z^n$ and $p\in\Z_+$.
By Claim \ref{cl-span}, $w^1\in\spn(\zA_{\ell})$ and hence, $w^2\in\spn(\zA_{\ell})$.
Since by hypothesis, $\zA_{\ell}$ is a DGSS, there exists a dyadic $\bar{z}$ for which,
\[
w^2=\sum_{i\in I(F_\ell)}\bar{z}_i\row_i(A).
\]
It follows from \eqref{w0w1} that,
\[
w^0=u\row_{\hat\i}(A)^\top+\sum_{i\in I(F_\ell)}\frac1{2^p}\bar{z}_i\row_i(A).
\]
Since $\bar{z}$ is dyadic, and by Claim \ref{cl-u}, $u$ is dyadic,
$w^0$ is a dyadic combinations of vectors in $\zA_\ell\subset\zA_{\ell-1}$ and of $\row_{\hat\i}(A)\in\zA_{\ell-1}$.
\qed
\end{proof}

\subsection{The proof}
\setcounter{claimcounter}{0}
We are now ready for the proof of our key result.
\begin{proof}[Theorem \ref{main-tech}]
Let $n=|E(\zC)|$ and let $p=|\zC|$. 
Let $M$ be the matrix where the rows correspond to the characteristic vectors of $\zC$.
Let $A$ be the matrix obtained from $M$ by adding as additional rows the identity matrix $I_n$.
Then $A$ is an $m\times n$ matrix where $m=p+n$.
Let $b$ be the vector with $m$ entries where $b_i=1$ if $i\in[p]$ and $b_i=0$ otherwise. Then $P(\zC)=\{x:Ax\geq b\}$.

Let $F_0=\bar{F}$ denote the $w$-optimal face of $P(\zC)$.
Let us choose $k\in\Z_+$ maximum for which $(F_0,F_1,\ldots,F_k)$ is a face-chain rooted at $\bar{F}$ with the property that
for every $\ell\in[k]$, there exists $\rho\in\Z^n$ and $\hat{\i}\in I(F_{\ell-1})\setminus I(F_\ell)$ 
such that properties (g1), (g2) are satisfied (such a chain always exists as we can pick $k=0$).
If $F_k$ is compliant then $(F_0,F_1,\ldots,F_k)$ is a well-marked face-chain and it follows from Theorem \ref{tool} that \eqref{eq-tool} has an optimal dyadic solution as required.
Furthermore, we leave it as an exercise to prove that \eqref{eq-tool} has an optimal dyadic solution if and only if \eqref{dual} does.
Thus we will assume that $F_k$ is not compliant, i.e., 
there exists $w'\in\Z_+^{E(\zC)}$ for which $F_k$ is contained in the $w'$-optimal face  $F'$ and \eqref{eq-tool} does not have an optimal dyadic solution for right-hand-side $w'$.
Since $\bar{F}\subseteq F_k$ we have $v\in F_k$.
If $B$ is a $w$-minimal cover, then $\chi(B)\in\bar{F}$.
Hence, $B\in F_k$, in particular $B$ is a $w'$-minimum cover.
It follows that $w'$ is minimal.
Then by Proposition \ref{small}, $\sigma(F_k,P(\zC))\geq 7$ and by hypothesis $F_k$ is $v$-compatible.
Since $w'$ is minimal none of the non-negativity constraints are tight for $F_k$. 
Hence, for some subset $\zS\subseteq\zC$ the tight constraints for $F_k$ are $x(S)=1$ for $S\in\zS$.
Let $\zL$ denote the set of faces of dimension $1$ containing $v$.

\vspace{0.05in}\noindent{\bf Claim.}
For every $S\in\zS$ there exists $L\in\zL$ and $\rho\in L\cap Z^n$ with 
\[
\sum_{j\in S}\rho_j\in\{2,3\}.
\]
\begin{proofofclaim}
We have $S\in\zS$ and let $H_S=\{x\in \R^n: x(S)=1\}$. 
Every $L\in\zL$ is either a line segment, or a ray.
Suppose there exists a ray $L\in\zL$ where $L\not\subset H_S$.
Since $\rec(P(\zC))=\R_+^{E(\zC)}$, $L=v+\R_+\1_j$ where $j\in S$.
But then $L,\rho=v+\1_j$ are as required.
Thus we may assume all rays of $\zL$ are contained in $H_S$.
Since $P(\zC)$ is full dimensional, it follows from Proposition~\ref{cone} that
there exists a line segment $L\in\zL$ where $L\not\subset H_S$.
Denote by $w$ the vertex of $L$ distinct from $v$.
Then $w\notin H_S$ and $w=\chi(B)$ for some $B\in b(\zC)$.
It follows that $|B\cap S|>1$.
Proposition \ref{3-neighbour} then implies that there exists $B'\in b(\zC)$
with $|B'\cap S|=3$ where $\chi(B')$ is adjacent to $v$.
Let $L'\in\zL$ denote the $1$-dimensional line segment containing $\chi(B')$.
Then $L',\rho=\chi(B')$ are as required.
\end{proofofclaim}
\noindent
Construct $\zL'\subseteq\zL$ as follows: 
for every $S\in\zS$, include the face $L$ in the Claim.
Observe that by construction $\zL'$ is $(F,v)$-saturating.
Since $F_k$ is $v$-compatible, there exists $L\in\zL'$ for which $L\cap F_{k+1}\setminus\{v\}\neq\emptyset$ where $F_{k+1}$ is some immediate superface of $F_k$.
$L$ arises in the Claim and consider the corresponding constraint $x(S)\geq 1$ and the corresponding $\rho\in L\cap Z^n$.
Denote by $\hat{\i}$ the index of $A$ corresponding to constraint $x(D)\geq 1$.
Then (g1), (g2) hold for $\ell=k+1$ and $(F_0,\ldots,F_{k+1})$ contradicts our choice of $k$.
\qed

\end{proof}
\section{Closing remarks.}\label{sec-closing}
In Proposition \ref{small} we proved Conjecture~\ref{corner-cn} for the binary case where the number of active dual variables is bounded by $6$.
In that case we actually prove that \eqref{dual} has a $\frac12$-integer solution. 
Suppose that we could improve the result from $6$ to some value $\alpha$.
It would immediately allow us to replace $7$ in Theorem~\ref{main-tech} by $\alpha+1$.
This in turn, using Proposition~\ref{reduction}, would allow us to replace $6$ by $\alpha$ in Theorem~\ref{main-th}.
So we in fact proved the following statement.
\begin{theorem}
If Conjecture~\ref{corner-cn} holds for all binary instances with at most $\alpha$ active constraints, then
it also holds for all binary instances where the corner is $(k,\ell)$-degenerate for $k(\ell+1)\leq\alpha$.
\end{theorem}
In other words, proving the conjecture for higher number of active constraints implies the conjecture for more degenerate corners.
\subsubsection*{Acknowledgments.} 
{\fontsize{9}{10.8}\selectfont 
We thank Ahmad Abdi, and Andr\'as Seb\H{o} for many lively discussions on the topic.
We are also grateful to Jim Geelen, for suggesting the proof of Proposition \ref{4-inter}.
This research was supported in part by Discovery Grants from the Natural Sciences and
Engineering Research Council (NSERC) of Canada.}
\subsubsection*{AI disclosure statement.} 
{\fontsize{9}{10.8}\selectfont 
The final draft of this manuscript was reviewed by Anthropic (Claude Fable 5.1, October 2026) and the authors edited the draft in response to such a review. The authors take full responsibility for all of the content of this paper.
\subsubsection*{Conflict of Interest.} 
{\fontsize{9}{10.8}\selectfont
 The authors declare that they have no competing interests.

\end{document}